\documentclass[12pt]{article}

\usepackage{amsmath}
\usepackage{amssymb}
\usepackage{mathrsfs}
\usepackage{amsfonts}
\usepackage{amsthm}
\usepackage[colorlinks,linkcolor=blue,anchorcolor=red,citecolor=blue]{hyperref}

\newtheorem{Thm}{Theorem}[section]
\newtheorem{Def}[Thm]{Definition}
\newtheorem{Lem}[Thm]{Lemma}
\newtheorem{Coro}[Thm]{Corollary}

\newtheorem{Prop}[Thm]{Proposition}

\newcommand{\1}{\mathbf{1}}
\newcommand{\R}{\mathbb{R}}
\newcommand{\Rd}{{\mathbb{R}^d}}

\renewcommand{\Re}{\text{Re}}

\newcommand{\F}{\mathscr{F}}
\newcommand{\N}{\mathbb{N}}

\usepackage{cite}

\allowdisplaybreaks
\renewcommand{\Re}{\text{Re}}
\renewcommand{\S}{\mathscr{S}}
\newcommand{\<}{\langle}
\renewcommand{\>}{\rangle}

\title{Dissipation and Semigroup on $H^k_n$: Non-cutoff Linearized Boltzmann Operator with Soft Potential}
\author{Dingqun DENG
	\thanks{Department of Mathematics, City University of Hong Kong 
e-mail: dingqdeng2-c@my.cityu.edu.hk}}

\begin{document}

\maketitle


\begin{abstract}
  In this paper, we find that the linearized collision operator $L$ of the non-cutoff Boltzmann equation with soft potential generates a strongly continuous semigroup on $H^k_n$, with $k,n\in\R$. In the theory of Boltzmann equation without angular cutoff, the weighted Sobolev space plays a fundamental role. The proof is based on pseudo-differential calculus and in general, for a specific class of Weyl quantization, the $L^2$ dissipation implies $H^k_n$ dissipation. This kind of estimate is also known as the G{\aa}rding's inequality.

\end{abstract}


\section{Introduction}
In this article, we are interested in proving that the linearized Boltzmann operator $L$, defined by \eqref{DefL}, can generate a strongly continuous semigroup on weighted Sobolev space $H^k_n$, defined by \eqref{DefHkn}. 
The main result of this paper are theorem \ref{Dissipation_Theorem} and \ref{Main2}.
Previous G{\aa}rding's inequality is on $L^2$, but that's not enough for generating a strongly continuous semigroup on $H^k_n$.
The main difficulty is to prove that $L$ is dissipative on $H^k_n$ and the invertibility of $\lambda I-L$ for some $\lambda>0$. Continuing the work by \cite{Global2019} and representing the linearized Boltzmann operator into pseudo-differential operator in section \ref{section3},
we can split the linearized collision operator as $L = -b^w+K$, where $-b^w$ is dissipative on $L^2$ while $K$ is bounded on $L^2$. So it suffices to analyze the behavior of Weyl quantization $b^w$ on $H^k_n$. Since the argument is based on pseudo-differential operator, our work can also be applied to a more general symbol class.

\subsection{Model and notations}
Consider the Boltzmann equation in $d$-dimension $(d\ge2)$:
\begin{align}
  F_t + v\cdot\nabla_x F = Q(F,F).
\end{align}
Here $F = F(x,v,t)$ is the distribution function of particles at position $x\in\Rd$ with velocity $v\in\Rd$ at time $t\ge 0$. $Q(F,G)$ is the bilinear collision operator defined for sufficiently smooth functions $F,G$ by
\begin{align}
  Q(F,G) := \int_{\Rd}\int_{S^{d-1}}B(v-v_*,\sigma) (F'_*G' - F_*G) \, d\sigma dv_*,
\end{align}
where $F'_* = F(x,v'_*,t)$, $G' = G(x,v',t)$, $F_* = F(x,v_*,t)$, $G = G(x,v,t)$
and $(v,v_*)$ are the velocities of two gas particles before collision while
$(v',v'_*)$ are the velocities after collision satisfying the following conservation laws of momentum and energy,
\begin{align*}
  v+v_*=v'+v'_*,\ \ |v|^2+|v_*|^2=|v'|^2+|v'_*|^2.
\end{align*}We use the so-called $\sigma$-representation, that is, for $\sigma\in \mathbf{S}^{d-1}$,
\begin{align*}
  v' = \frac{v+v_*}{2}+\frac{|v-v_*|}{2}\sigma,\ \
  v'_* = \frac{v+v_*}{2}-\frac{|v-v_*|}{2}\sigma.
\end{align*}
and define the angle $\theta$ in the standard way
\begin{align*}
  \cos\theta = \frac{v-v_*}{|v-v_*|}\cdot \sigma,
\end{align*}where $\cdot$ denotes the usual inner product in $\R^d$.
The collision kernel $B$ satisfies
\begin{align}
  B(v-v_*,\sigma) = |v-v_*|^\gamma b(\cos\theta),
\end{align}
for some $\gamma\in\R$ and function $b$. Without loss of generality, we can assume $B(v-v_*,\sigma)$ is supported on $(v-v_*)\cdot\sigma\ge 0$ which corresponds to $\theta\in[0,\pi/2]$, since $B$ can be replaced by its symmetrized form $\overline{B}(v-v_*,\sigma) = B(v-v_*,\sigma)+B(v-v_*,-\sigma)$.
Moreover, we are going to work on the collision kernel without angular cut-off, which corresponds to the case of inverse power interaction laws between particles. That is,
\begin{align}
  b(\cos\theta)\approx \theta^{-d+1-2s}\ \text{ on }\theta \in (0,\pi/2).
\end{align}
Here we assume
\begin{align}
  s\in (0,1), \quad \gamma\in (-d,\infty).
\end{align}For Boltzmann equation without angular cut-off, the condition $\gamma+2s\le 0$ is called soft potential while $\gamma+2s>0$ is called hard potential. The behavior of this kernel gives non-integrability condition
\begin{align*}
  \int^{\pi/2}_0\sin^{d-2}\theta\,b(\cos\theta)\,d\theta=\infty,
\end{align*}which becomes the major difficulty in the theory of Boltzmann equation without angular cut-off.

We are looking for a solution $f$ near the normalized equilibrium, which is the normalized global Maxwellian
\begin{align*}
  \mu(v)=(2\pi)^{-d/2}e^{-|v|^2/2}.
\end{align*}
Set $F = \mu + \mu^{\frac{1}{2}}f$. Then the perturbation $f$ satisfies
\begin{align*}
  f_t + v\cdot\nabla_x f = Lf + \mu^{-1/2}Q(\mu^{1/2}f,\mu^{1/2}f),
\end{align*}
where $L$ is called the linearized Boltzmann operator defined by
\begin{align}\label{DefL}
  Lf := \mu^{-1/2}Q(\mu,\mu^{1/2}f) + \mu^{-1/2}Q(\mu^{1/2}f,\mu).
\end{align}

One may refer to \cite{Cercignani1994,Alexandre2012,Alexandre2001,Gressman2011} for more introduction on the mathematical theory of Boltzmann equation.
In \cite{Ukai1982,Deng2018}, the weighted $L^2$ space is necessary for the analysis to Boltzmann equation with angular cut-off and soft potential, since the estimate for nonlinear term $\mu^{-1/2}Q(\mu^{1/2}f,\mu^{1/2}f)$ in the equation is on the weighted $L^2$ space.
While for the non-cutoff case, this kind of estimate only valid in the weighted Sobolev space $H^k_n$, for instance \cite{Gressman2011,Alexandre2000,Alexandre2012}, since non-cutoff Boltzmann equation essentially requires derivative.

Assume $k,n\in\R$ and define the weighted Sobolev space $H^k_n(\Rd )$ by
\begin{align*}
  H^k_n(\Rd):=\{f\in\S'(\Rd):\|f\|_{H^k_n}<\infty\},
\end{align*}where
\begin{align}\label{DefHkn}
  \|f\|_{H^k_n}:=\|\<\eta\>^k\F(\<\cdot\>^nf)\|_{L^2},
\end{align}where $\F$ is the Fourier transform on $\Rd$: $\F f(\eta):=\int_\Rd f(v)e^{2\pi i v\cdot\eta}\,dv$. For later use, we define
\begin{align}
  c(v,\eta):=\<v\>^n\<\eta\>^k.
\end{align}Then $c$ is a $\Gamma$-admissible weight function as well as a symbol in $S(c)$, with $\Gamma=|dv|^2+|d\eta|^2$. One may refer to the appendix as well as \cite{Lerner2010,Bony1998-1999,Beals1981,Bony1994} for more information about pseudo-differential calculus.
In the corollary \ref{equaivlent_coro} below, we can prove that
\begin{align}
  \|c^w(v,D_v)f\|_{L^2}\approx\|\<v\>^n\<D_v\>^kf\|_{L^2}\approx \|\<D_v\>^k\<v\>^nf\|_{L^2}.
\end{align}Thus the space $(H(c),\|\cdot\|_{H(c)})$ is equivalent to $(H^k_n,\|\cdot\|_{H^k_n})$. So we don't distinguish this two spaces below and will equip $H^k_n$ with norm $\|\cdot\|_{H(c)}=\|c^w(\cdot)\|_{L^2}$.

\paragraph{Notations} Throughout this article, we shall use the following notations.
For any $v\in\Rd$, we denote $\<v\>=(1+|v|^2)^{1/2}$. The gradient in $v$ is denoted by $\partial_v$. Also we use notation $D_v=\frac{\partial_v}{i}$ and $\<D_v\>^kf=\F^{-1}(\<\cdot\>^k\F f)$. Let $A\in\Rd$, denote $\1_{A}$ to be the characteristic function that equal to $1$ on $A$ and $0$ on $\Rd\setminus A$. $L(X)$ is the space of all linear continuous operator on Banach space $X$. 

The notation $a\approx b$ (resp. $a\gtrsim b$, $a\lesssim b$) for positive real function $a$, $b$ means there exists $C>0$ not depending on possible free parameters such that $C^{-1}a\le b\le Ca$ (resp. $a\ge C^{-1}b$, $a\le Cb$) on their domain. $\Re (a)$ means the real part of complex number $a$.

For pseudo-differential calculus, we write $\Gamma=|dv|^2+|d\eta|^2$ to be an admissible metric.
Let $m$, $l$ be two $\Gamma$-admissible weight functions and write $S(m):=S(m,\Gamma)$, $H(m):=H(m,\Gamma)$, $a_{K,l}:=a+Kl$. $a^w$ is the Weyl quantization.

\subsection{Main results}

Our first result is on general symbols. We find that the $L^2$ dissipation of Weyl quantization $a^w(v,D_v)$ can imply the $H^n_k$ dissipation.

\begin{Thm}\label{Dissipation_Theorem}
    Let $m$, $l$ be two $\Gamma$-admissible weights, $\rho>0$, $\varepsilon\in(0,1)$. Assume $l\in S(l)$, $l\lesssim m$, $m\<\eta\>^{-N}\lesssim l$ for some $N>0$ and

    (1). $a\in S(m)$, $\partial_\eta a\in S(\varepsilon m_{K,l}+\varepsilon^{-\rho}l)$ uniformly in $\varepsilon$.

    (2). $b^{1/2}\in S(m^{1/2})$, $\partial_\eta (b^{1/2})_{K,l^{1/2}}\in S(K^{-\kappa}(m^{1/2})_{K,l^{1/2}})$ uniformly in $K$ and\\ $(b^{1/2})_{K,l^{1/2}}\gtrsim (m^{1/2})_{K,l^{1/2}}$.

    (3). Suppose for $f\in\S$,
  \begin{align}\label{eqq4}
    \Re(a^w(v,D_v)f,f)_{L^2}\ge \frac{1}{C}\|(b^{1/2})^w(v,D_v)f\|^2_{L^2} - C\|(l^{1/2})^wf\|^2_{L^2},
  \end{align}for some constant $C$ independent of $f$.
  Then for $k,n\in\R$, $f\in\S$,
  \begin{align}\label{eqq5}
    \Re(a^w(v,D_v)f,f)_{H^k_n}
    &\ge \frac{1}{C'}\|(b^{1/2})^w(v,D_v)c^wf\|^2_{L^2} - C_{k}\|(l^{1/2})^wc^wf\|^2_{L^2},
  \end{align}for some $C', C_{k}>0$.
  \end{Thm}

The assumption on $a$ essentially represents the smallness on $\partial_\eta a$, which can be viewed as a general version of \eqref{varepsilon_inequality}.
Although there are a lot of restriction on symbol $b$, in our application to Boltzmann equation, we can choose $b=m$. Then these assumptions are trivial for checking. The real part $\Re$ in \eqref{eqq4}\eqref{eqq5} can be replaced by imaginary part, since they don't have essential difference.
Also the symbol $c$ can be generalized to a symbol class that $\partial_v c$ and $\partial_\eta c$ have better decay on direction $v$, $\eta$ respectively.

The main idea is based on controlling the commutator $[c^w,a^w]$. Once we get the estimate on it, we can evaluate the difference between $\Re(a^wc^wf,c^wf)_{L^2}$ and $\Re(a^wf,f)_{H^k_n}$. Then we can have the dissipation on $H^k_n$ from $L^2$.

As an application, we can prove our result on strongly continuous semigroup.
Define
\begin{align}
  \tilde{a}(v,\eta):=\<v\>^\gamma(1+|\eta|^2+|\eta\wedge v|^2+|v|^2)^s.
\end{align}
Then $\tilde{a}$ is a $\Gamma$-admissible weight proved in \cite{Global2019}.

\begin{Thm}\label{Main2}Assume $\gamma+2s\le 0$. There exists $C_1>1$ such that the linearized Boltzmann operator
  $L$ generates a strongly continuous semigroup on $H(c)=H^k_n(\Rd)$ with domain
  $D(L):=H((\tilde{a}+C_1)c)$.
\end{Thm}
The constant $C_1>1$ here is to ensure the continuous embedded: $H((\tilde{a}+C_1)c)\hookrightarrow H(c)$.
The linearized Boltzmann operator $L$ can be splitted as
\begin{align}
  L=-b^w+K.
\end{align}So once we apply the theorem \ref{Dissipation_Theorem} to symbol $b$, theorem \ref{Main2} follows from the boundedness of $K$. Also we can prove that $K$ can be written as a pseudo-differential operator with symbol in $S(1)$, thus $K$ is bounded on $H((\tilde{a}+C_1)c)$ and hence on $H(c)$.

\paragraph{Organization of the article}

The paper is organized as follows. In Section 2, we provide the dissipation on $H^k_n$ and have a discussion on general symbol class on $\Rd$. Some useful lemmas in pseudo-differential calculus are provided.
In Section 3, we deal with the linearized Boltzmann operator $L=-b^w+K$ on $\R^d$,
where Carleman representation is applied from time to time.
An appendix is devoted to a short review of some useful tools used in this work such as pseudo-differential calculus and semigroup theory.

\section{Dissipation on $H^k_n$}
In this section, we are going to prove that the $L^2$ dissipation implies $H^k_n$ dissipation. Here, we fix $k,n\in\R$, $\kappa>0$ and consider $\R^d_v$ to be the whole space.
Let $m$, $l$ be two $\Gamma$-admissible weight functions.
Recall that $a_{K,l}:=a+Kl$, $m_{K,l}:=m+Kl$ for $K>1$.
Also, we always assume in this section that
\begin{align}\label{assumption_l}
	l\in S(l)\ \text{ and }\ l\lesssim m.
\end{align}
We should remind readers that the lemma \ref{inverse_lemma}, \ref{inverse_lemma_2}, \ref{inverse_bounded_lemma}, \ref{bound_varepsilon} below are valid for any $\Gamma$-admissible metric $c$, which will be used later.
\begin{Lem}\label{inverse_lemma}
  Assume $a\in S(m)$, $\partial_\eta (a_{K,l})\in S(K^{-\kappa}m_{K,l})$ uniformly in $K$ and
  $|a_{K,l}|\gtrsim m_{K,l}$. Then

   (1). $a^{-1}_{K,l}\in S(m^{-1}_{K,l})$, uniformly in $K$, for $K>1$.

  (2). There exists $K_0>1$ sufficiently large such that for all $K>K_0$, $a^w_{K,l}:H(mc)\to H(c)$ is invertible and its inverse $(a^w_{K,l})^{-1}: H(c) \to H(mc)$ satisfies
  \begin{align}\label{inverse_equation}
    (a^w_{K,l})^{-1} = G_{1,K,l}(a^{-1}_{K,l})^w = (a^{-1}_{K,l})^wG_{2,K,l},
  \end{align}where $G_{1,K,l}\in L(H(mc))$, $G_{2,K,l}\in L(H(c))$ with operator norm smaller than $2$.
\end{Lem}
\begin{proof}
  Since $l\in S(l)\subset S(m)$, we have $a_{K,l}\in S(m_{K,l})\subset S(m)$ and so $a_{K,l}$ maps $H(mc)$ continuously into $H(c)$.
  By composition formula of Weyl quantization,
  \begin{align}
    a^w_{K,l}(a^{-1}_{K,l})^w=I+R^w_{K,l},
  \end{align}where
  \begin{align}
    R_{K,l}=\int^1_0(\partial_{v}a_{K,l}\#_\theta \partial_{\eta} a^{-1}_{K,l}-\partial_{\eta} a_{K,l}\#_\theta \partial_{v} a^{-1}_{K,l})\,d\theta.
  \end{align}
  For any $1\le j\le d$,
    \begin{align*}
      \partial_{\eta_j} a^{-1}_{K,l} = -\frac{\partial_{\eta_j}a_{K,l}}{a_{K,l}^2},\quad
      |\partial_{\eta_j} a^{-1}_{K,l}|\lesssim \frac{K^{-\kappa}m_{K,l}}{m_{K,l}^2}
      \lesssim \frac{K^{-\kappa}}{m_{K,l}},
    \end{align*}
    Estimate on higher derivative follows from Leibniz formula.
    Thus we have $\partial_{\eta} a^{-1}_{K,l} \in S(K^{-\kappa}m_{K,l}^{-1})$ and $\partial_{\eta} a_{K,l}\in S(K^{-\kappa}m_{K,l})$ uniformly in $K$.
    Similarly, $\partial_{v} a^{-1}_{K,l}\in S(m_{K,l}^{-1})$ and by definition, $\partial_{v}a_{K,l}\in S(m_{K,l})$ uniformly in $K$.
  Applying \ref{sharp_theta}, for any $N\in\N$, there exists $l_N\in\N$ independent of $K$ and $\theta$ such that
  \begin{align*}
    \|\partial_{v}a_{K,l}\#_\theta \partial_{\eta} a^{-1}_{K,l}\|_{N;S(1)}\le C_N\|\partial_{v}a_{K,l}\|_{l_N;S(m_{K,l})}\|\partial_{\eta} a^{-1}_{K,l}\|_{l_N;S(m^{-1}_{K,l})}\le C'_NK^{-\kappa}.
  \end{align*}Similarly,
  \begin{align*}
    \|\partial_{\eta} a_{K,l}\#_\theta \partial_{v} a^{-1}_{K,l}\|_{N;S(1)}\le C'_NK^{-\kappa}.
  \end{align*}
  Thus $\{K^{\kappa}\partial_{\eta} a_{K,l}\#_\theta \partial_{v} a^{-1}_{K,l}\}$ and $\{K^{\kappa}\partial_{v} a_{K,l}\#_\theta \partial_{\eta} a^{-1}_{K,l}\}$
  are uniformly bounded sets in $S(1)$ with respect to $K$ and $\theta$.
  Thus by Remark 3.4 in \cite{Beals1981}, the operator $K^{\kappa}R^w_{K,l}$ is linear continuous on $H(mc)$ and $H(c)$ with operator norm independent of $K$.
  So there exists $K_0>1$ such that for $K>K_0$,
  \begin{align*}
    I+K^{-\kappa}(K^{\kappa}R^w_{K,l})
  \end{align*}is invertible on $H(mc)$ and $H(c)$ and the operator norm of inverse $(I+R^w_{K,l})^{-1}$ on $H(mc)$ and $H(c)$ are smaller than $2$. Thus
  \begin{align*}
    a^w_{K,l}(a^{-1}_{K,l})^w(I+R^w_{K,l})^{-1} = I\quad \text{on $H(c)$.}
  \end{align*}
  Similarly, by choosing $K_0$ sufficiently large, we can find $\tilde{R}_{K,l}\in S(1)$ such that $(I+\tilde{R}^w_{K,l})^{-1}$ is invertible on $H(mc)$ whenever $K>K_0$ and
  \begin{align*}
    (I+\tilde{R}^w_{K,l})^{-1}(a^{-1}_{K,l})^wa^w_{K,l} = I\quad \text{on $H(mc)$.}
  \end{align*}
  Noticing $a^{-1}_{K,l}\in S(m^{-1})$ and $(a^{-1}_{K,l})^w$ maps $H(c)$ continuously into $H(mc)$, we obtained that $a^w_{K,l}:H(mc)\to H(c)$ has left inverse and right inverse, and hence is invertible with inverse in the form of \eqref{inverse_equation}.
\end{proof}

Notice that in this lemma, the symbol $a$ may not be real-valued. This is necessary in next section. For further application, we state a similar lemma on $a^{1/2}_{K,l}$, which needs $a$ to be positive.

\begin{Lem}\label{inverse_lemma_2}
  Assume $a\in S(m)$, $\partial_\eta (a_{K,l})\in S(K^{-\kappa}m_{K,l})$ uniformly in $K$ and
  $a_{K,l}\gtrsim m_{K,l}$.\\
  Then (1). $a^{1/2}_{K,l}\in S(m^{1/2}_{K,l})$, $a^{-1/2}_{K,l}\in S(m^{-1/2}_{K,l})$, uniformly in $K$, for $K>1$.\\
  (2). There exists $K_0>1$ sufficiently large such that for all $K>K_0$, $(a^{1/2}_{K,l})^w:H(m^{1/2}c)\to H(c)$ is invertible and its inverse $((a^{1/2}_{K,l})^w)^{-1}: H(c) \to H(m^{1/2}c)$ satisfies
  \begin{align}\label{inverse_equation_2}
	((a^{1/2}_{K,l})^w)^{-1} = F_{1,K,l}(a^{-1/2}_{K,l})^w = (a^{-1/2}_{K,l})^wF_{2,K,l},
  \end{align}where $F_{1,K,l}\in L(H(m^{1/2}c))$, $F_{2,K,l}\in L(H(c))$ with operator norm small than $2$.
\end{Lem}
\begin{proof}
  Firstly by assumption on $a$ and $l$, $a_{K,l}\in S(m_{K,l})$ uniformly in $K$.
  Similar to lemma \ref{inverse_lemma}, we have for any $1\le j\le d$,
  \begin{align*}
    \partial_{\eta_j}a^{1/2}_{K,l} = \frac{\partial_{\eta_j}a_{K,l}}{2a^{1/2}_{K,l}},\
    |\partial_{\eta_j}a^{1/2}_{K,l}| = \frac{|\partial_{\eta_j}a_{K,l}|}{2|a^{1/2}_{K,l}|} \lesssim K^{-\kappa}m^{1/2}_{K,l}.
  \end{align*}
  Estimate on higher derivative follows from Leibniz formula and thus we have
  $\partial_{\eta} a^{1/2}_{K,l}\in S(K^{-\kappa}m^{1/2}_{K,l})$ uniformly in $K$.
  Similarly, $\partial_{\eta} a^{-1/2}_{K,l} \in S(K^{-\kappa}m^{-1/2}_{K,l})$,
  $\partial_{v}a^{1/2}_{K,l}\in S(m^{1/2}_{K,l})$ and $\partial_{v} a^{-1/2}_{K,l}\in S(m^{1/2}_{K,l})$ uniformly in $K$.

  By composition formula of Weyl quantization,
  \begin{align*}
    (a^{1/2}_{K,l})^w(a^{-1/2}_{K,l})^w=I+R^w_{K,l},
  \end{align*}where
  \begin{align*}
    R_{K,l}=\int^1_0(\partial_{v}a^{1/2}_{K,l}\#_\theta \partial_{\eta} a^{-1/2}_{K,l}-\partial_{\eta} a^{1/2}_{K,l}\#_\theta \partial_{v} a^{-1/2}_{K,l})\,d\theta.
  \end{align*}
  Thus the following argument is exactly the same as lemma \ref{inverse_lemma} and we omit them.
\end{proof}

\begin{Lem}\label{inverse_bounded_lemma}Let $m$ be $\Gamma$-admissible weight such that $a\in S(m)$.
  Assume $a^w:H(mc)\to H(c)$ is invertible.
    If $b\in S(m)$, then there exists $C>0$, depending only on the seminorms of $a$ and $b$, such that for $f\in H(mc)$,
    \begin{align}
      \|b(v,D_v)f\|_{H(c)}+\|b^w(v,D_v)f\|_{H(c)}\le C\|a^w(v,D_v)f\|_{H(c)}.
    \end{align}
\end{Lem}
\begin{proof}
  Applying Corollary 2.6.28 in \cite{Lerner2010}, we know that there exists $a_{-1}\in S(m^{-1})$ such that
  $a\#b=b\#a=1$. Thus $a_{-1}^wa^w=I$ on $H(mc)$.
  Since $b\in  S(m)$, we have $b\#a_{-1}\in S(1)$
  and hence $b^wa_{-1}^w$ is a linear bounded operator on $H(c)$. Thus
  \begin{align*}
    b^w = b^wa_{-1}^w a^w \quad\text{on $H(mc)$,}
  \end{align*}and so for $f\in H(mc)$,
  \begin{align*}
    \|b^w(v,D_v)f\|_{H(c)}\le C_{K,l}\|a^w(v,D_v)f\|_{H(c)}.
  \end{align*}
  On the other hand, $b(v,D_v)=(J^{-1/2}b)^w$ and $J^{-1/2}b\in S(m)$, thus $b(v,D_v)$ has the same bound as $b^w(v,D_v)$.

\end{proof}

In lemma \ref{inverse_lemma}, we obtained that $(a_{K+\varepsilon^{-(1+\kappa)},l})^w$ is invertible for sufficiently large $K$. Hence the following corollary is a similar result to lemma \ref{inverse_bounded_lemma} but the proof is slightly different, since $b$ belongs to a different symbol class and we need the constant to be independent of $\varepsilon$.

\begin{Lem}\label{bound_varepsilon}
  Assume $a\in S(m)$, $\partial_\eta (a_{K,l})\in S(K^{-\kappa}m_{K,l})$ uniformly in $K$ and
  $a_{K,l}\gtrsim m_{K,l}$.
  Let $\rho>0$ and $b\in S(\varepsilon m_{K,l}+\varepsilon^{-\rho}l)$, uniformly in $\varepsilon\in(0,1)$. Then there exists $K_0>0$, such that for $f\in H(mc)$, $\varepsilon\in(0,1)$,
      \begin{align}
        \|b(v,D_v)f\|_{H(c)}+\|b^w(v,D_v)f\|_{H(c)}\le C_{K,l}\left(\varepsilon\|a^w(v,D_v)f\|_{H(c)}+\varepsilon^{-\rho}\|l^wf\|_{H(c)}\right).
      \end{align}
\end{Lem}
\begin{proof}
  From lemma \ref{inverse_lemma}, we have $a^{-1}_{K,l}\in S(m^{-1}_{K,l})$ for $K>1$, and there exists $K_0>1$ such that for $K>K_0$,
  \begin{align}
    (a^w_{K,l})^{-1} = (a^{-1}_{K,l})^wG_{2,K,l},
  \end{align}with $G_{2,K,l}\in L(H(c))$.
  Since $b\in S(\varepsilon m_{K+\varepsilon^{-1-\rho},l})$, we have $\varepsilon^{-1} b\#a^{-1}_{K+\varepsilon^{-1-\rho},l}\in S(1)$, uniformly in $\varepsilon$.
  Write
  \begin{align}
    b^w=\underbrace{\varepsilon^{-1} b^w(a^{-1}_{K+\varepsilon^{-1-\rho},l})^wG_{2,K,l}}_{\text{bounded}}\,\varepsilon\,(a_{K+\varepsilon^{-1-\rho},l})^w,
  \end{align} then
  \begin{align}
    \|b^w(v,D_v)f\|_{H(c)}\le C_{K,l,d}\,\varepsilon\|a_{K+\varepsilon^{-1-\rho},l}^w(v,D_v)f\|_{H(c)}.
  \end{align}
  Similar to the previous lemma, applying that $b(v,D_v)=(J^{-1/2}b)^w$ and $J^{-1/2}b\in S(\varepsilon m_{K,l}+\varepsilon^{-\kappa}l)$, we have $b(v,D_v)$ has the same bound as $b^w$.
\end{proof}

For $k,n\in\R$, it's trivial to obtain that $\<v\>^{n}$ and $\<D_v\>^k$, as Weyl quantization, are invertible, since $\<v\>^n$ is only a multiplication while $\<D_v\>^k$ is a multiplier.
Recall
\begin{align}
  c = \<\eta\>^k\<v\>^n.
\end{align}Then $c\in S(\<\eta\>^k\<v\>^n)$, $\partial_v c\in S(\<\eta\>^k\<v\>^{n-1})$, $\partial_\eta c\in S(\<\eta\>^{k-1}\<v\>^{n})$ and we have the following useful corollary.
There are many ways to prove this corollary, here we provide one by applying the above lemmas.
\begin{Coro}\label{equaivlent_coro}Let $k,n\in\R$, then we have the equivalence
  \begin{align}
    \|c^w(v,D_v)f\|_{L^2}\approx\|\<v\>^n\<D_v\>^kf\|_{L^2}\approx \|\<D_v\>^k\<v\>^nf\|_{L^2},
  \end{align}and hence these norms are equivalent on $H^k_n$.
\end{Coro}
\begin{proof}
  The symbols of $\<v\>^n\<D_v\>^k\<v\>^{-n}$ and $c^w\<v\>^{-n}$ belong to $S(\<\eta\>^k)$.
  Letting $m = \<\eta\>^k$ in lemma \ref{inverse_bounded_lemma}, we find that for $f\in H(\<\eta\>^k)$,
  \begin{align*}
    \|\<v\>^n\<D_v\>^k\<v\>^{-n}f\|_{L^2} &\lesssim\|\<D_v\>^kf\|_{L^2},\\
    \|c^w(v,D_v)\<v\>^{-n}f\|_{L^2} &\lesssim\|\<D_v\>^kf\|_{L^2}.
  \end{align*}
  So for any $f$ such that $\|\<D_v\>^k\<v\>^nf\|_{L^2}<\infty$,
  \begin{align*}
    \|\<v\>^n\<D_v\>^kf\|_{L^2} &\lesssim \|\<D_v\>^k\<v\>^nf\|_{L^2},\\
    \|c^w(v,D_v)f\|_{L^2} &\lesssim \|\<D_v\>^k\<v\>^nf\|_{L^2}.
  \end{align*}
  Similarly, the symbol of $\<D_v\>^k\<v\>^n\<D_v\>^{-k}$ belonges to $S(\<v\>^n)$.
  Letting $m = \<v\>^n$ in lemma \ref{inverse_bounded_lemma}, we find that for any $f\in L^2_x$,
  \begin{align*}
    \|\<D_v\>^k\<v\>^n\<D_v\>^{-k}f\|_{L^2} &\lesssim \|\<v\>^nf\|_{L^2},
  \end{align*}and for $f$ such that $\|\<v\>^n\<D_v\>^kf\|_{L^2}<\infty$,
  \begin{align*}
    \|\<D_v\>^k\<v\>^nf\|_{L^2} &\lesssim \|\<v\>^n\<D_v\>^kf\|_{L^2}.
  \end{align*}So we proved the second equivalence.
  Besides, the symbol $c$ satisfies lemma \ref{inverse_lemma} with $m=l=\<\eta\>^k\<v\>^n$. Thus there exists $K>1$ such that $c^w_{K,l}=(K+1)c^w:H(\<\eta\>^k\<v\>^n)\to L^2$ is invertible in the form of \eqref{inverse_equation}.
  Thus applying lemma \ref{inverse_bounded_lemma} and $\<D_v\>^k\<v\>^n\in S(\<\eta\>^k\<v\>^n)$, we have
  \begin{align*}
    \|\<D_v\>^k\<v\>^nf\|_{L^2}\lesssim \|c^w(v,D_v)f\|_{L^2}.
  \end{align*}
\end{proof}



Here we give a version of G{\aa}rding's inequality, which is needed in the next section.
\begin{Thm}\label{Dissipation_theoremL2}
  Assume $a\in S(m)$, $\partial_\eta (a_{K,l})\in S(K^{-\kappa}m_{K,l})$ uniformly in $K$ and $a_{K,l}\gtrsim m_{K,l}$.
  Then there exists $K_0>1$ such that for $K>K_0$, $f\in\S$, we have
  \begin{align}\label{equaivalent_1}
    \Re(a_{K,l}^w(v,D_v)f,f)_{L^2}\approx \|(a^{1/2}_{K,l})^wf\|_{L^2}^2,
  \end{align}
  If in addition, $b^{1/2}\in S(m^{1/2})$, then
    \begin{align}\label{Dissipation_L2}
      \Re(a^w(v,D_v)f,f)_{L^2}\ge \frac{1}{C}\|(b^{1/2})^w(v,D_v)f\|^2_{L^2} - C(l^w(v,D_v)f,f)_{L^2},
    \end{align}for some constant $C$ independent of $f$.
\end{Thm}
\begin{proof}
  Notice $a$ satisfies the assumption of lemma \ref{inverse_lemma_2},
  thus there exists $K_0>1$ such that for $K>K_0$, $(a^{1/2}_{K,l})^w: H(m^{1/2})\to L^2$ is invertible with formula \eqref{inverse_equation_2}. Hence by lemma \ref{inverse_bounded_lemma} and $b^{1/2}\in S(m^{1/2})$, we have for $f\in H(m^{1/2})$,
  \begin{align}\label{tempeq1}
    \|(b^{1/2})^wf\|_{L^2}\le C_K \|(a^{1/2}_{K,l})^wf\|_{L^2}.
  \end{align}

  On the other hand,
  \begin{align*}
    (a^{1/2}_{K,l})^w(a^{1/2}_{K,l})^w = a_{K,l}^w + R_{K,l}^w,
  \end{align*}with
  \begin{align*}
    R_{K,l}^w=\int^1_0(\partial_{v}a^{1/2}_{K,l}\#_\theta \partial_{\eta} a^{1/2}_{K,l}-\partial_{\eta} a^{1/2}_{K,l}\#_\theta \partial_{v} a^{1/2}_{K,l})\,d\theta.
  \end{align*}
  Similar to the proof in lemma \ref{inverse_lemma}, since $\partial_{v}a^{1/2}_{K,l}\in S(m^{1/2}_{K,l})$ and $\partial_{\eta}a^{1/2}_{K,l}\in S(K^{-\kappa}m^{1/2}_{K,l})$ uniformly in $K$,
  we have $\partial_{v}a^{1/2}_{K,l}\#_\theta \partial_{\eta} a^{1/2}_{K,l}$ and $ \partial_{\eta} a^{1/2}_{K,l}\#_\theta \partial_{v} a^{1/2}_{K,l}$ belong to $S(K^{-\kappa}m_{K,l})$ uniformly in $K$ and $\theta$.
  Hence $R_{K,l}\in S(K^{-\kappa}m_{K,l})$ uniformly in $K$. Using \eqref{inverse_equation_2}, 
  \begin{align*}
    R^w_{K,l} = K^{-\kappa}(a^{1/2}_{K,l})^wF_{1,K,l}\left(K^{\kappa}(a^{-1/2}_{K,l})^wR_{K,l}^w(a^{-1/2}_{K,l})^w\right)F_{2,K,l}(a^{1/2}_{K,l})^w,
  \end{align*}where $g^w:=K^{\kappa}(a^{-1/2}_{K,l})^wR_{K,l}^w(a^{-1/2}_{K,l})^w$ has symbol in $S(1)$ uniformly in $K$, hence is bounded on $L^2$. Recall that the norm of operators $F_{1,K,l}$ and $F_{2,K,l}$ are smaller than $2$. Then for $f\in\S$,
  \begin{align*}
    (R^w_{K,l}f,f)_{L^2}& = K^{-\kappa}\left(F_{1,K,l}g^wF_{2,K,l}(a^{1/2}_{K,l})^wf,(a^{1/2}_{K,l})^wf\right)_{L^2},\\
    |(R^w_{K,l}f,f)_{L^2}|&\le K^{-\kappa}C\|(a^{1/2}_{K,l})^wf\|_{L^2}^2.
  \end{align*}
  We choose $K_0$ sufficiently large such that for $K>K_0$,
  \begin{align*}
    |(R^w_{K,l}f,f)_{L^2}|\le \frac{1}{2}\|(a^{1/2}_{K,l})^wf\|^2_{L^2}.
  \end{align*}
  Then for $f\in\S$,
  \begin{align*}
    \|(a^{1/2}_{K,l})^wf\|^2_{L^2}& = (a_{K,l}^wf,f)_{L^2} +(R_{K,l}^wf,f)_{L^2},\\
    \|(a^{1/2}_{K,l})^wf\|^2_{L^2}&\approx \Re(a_{K,l}^wf,f)_{L^2}.
  \end{align*}
  Together with \eqref{tempeq1}, we get \eqref{Dissipation_L2}.
\end{proof}

Now we come to the proof of our main theorem \ref{Dissipation_Theorem}.

\begin{proof}[Proof of theorem \ref{Dissipation_Theorem}]

We claim that for any $k,n\in\R$, there exists constant $C_{k,n}>0$ such that for $\varepsilon>0$, $f\in\S$,
\begin{align}
  |(a^wf,f)_{H^k_n} - (a^wc^wf,c^wf)_{L^2}|\le \varepsilon\|(b^{1/2})^wc^wf\|_{L^2}^2 + C_{k,\varepsilon}\|(l^{1/2})^wc^wf\|^2_{L^2}.
\end{align}
Then letting $\varepsilon$ small, we have for $f\in\S$,
\begin{align*}
  \Re(a^wf,f)_{H^k_n}
  &\ge \Re(a^wc^wf,c^wf)_{L^2}-(\varepsilon\|(b^{1/2})^wc^wf\|_{L^2}^2 + C_{k,\varepsilon}\|(l^{1/2})^wc^wf\|^2_{L^2})\\
  &\ge \frac{1}{C}\|(b^{1/2})^wc^wf\|^2_{L^2}-\varepsilon\|(b^{1/2})^wc^wf\|_{L^2}^2 - C_{k,\varepsilon}\|(l^{1/2})^wc^wf\|^2_{L^2}\\
  &\ge \frac{1}{C'}\|(b^{1/2})^wc^wf\|^2_{L^2} - C_{k}\|(l^{1/2})^wc^wf\|^2_{L^2}.
\end{align*}
Proof of the claim: Notice
    \begin{align*}
      (a^w(v,D_v)f,f)_{H^k_n}
      &=(c^w(v,D_v)a^w(v,D_v)f,c^w(v,D_v)f)_{L^2}\\
      &=(a^w(v,D_v)c^w(v,D_v)f,c^w(v,D_v)f)_{L^2}+([c^w,a^w]f,c^wf)_{L^2}.
    \end{align*}
    So it suffices to control the last term.
  Since $b^{1/2}$ satisfies the assumptions of lemma \ref{inverse_lemma}, there exists $K_0>1$ such that for $K>K_0$, $(b^{1/2})_{K,l^{1/2}}\in S((m^{1/2})_{K,l^{1/2}})$ is invertible in the form of \eqref{inverse_equation} and $((b^{1/2})_{K,l^{1/2}})^{-1}\in S(((m^{1/2})_{K,l^{1/2}})^{-1})$.
  Noticing $c=\<\eta\>^k\<v\>^n$, $c^{-1}\in S(\<\eta\>^{-k}\<v\>^{-n})$, $\partial_\eta a\in S(\varepsilon m_{K,l}+\varepsilon^{-\rho}l)$,
  $m\<\eta\>^{-N}\lesssim l$ for some $N>0$
  and \eqref{varepsilon_inequality}, we have, for any $0<\varepsilon<1$,
  \begin{align*}
    &[c^w(v,D_v),a^w(v,D_v)]\in Op((\varepsilon m_{K,l}+\varepsilon^{-\delta}l)\<\eta\>^{k}\<v\>^n),\\
    &[c^w(v,D_v),a^w(v,D_v)](c^{-1})^w\in Op(\varepsilon m_{K,l}+\varepsilon^{-\delta}l),
  \end{align*}for some $\delta>0$. Thus fixing $K>K_0$,
  \begin{align*}
    g^w:=(((b^{1/2})_{(K+\varepsilon^{-1-\delta})^{1/2},l^{1/2}})^{-1})^w[c^w,a^w](c^{-1})^w
  \end{align*}
   has a symbol in $S\big(\varepsilon (m^{1/2}+(K+\varepsilon^{-1-\delta})^{1/2}l^{1/2})\big)$ uniformly in $\varepsilon$ and hence by lemma \ref{bound_varepsilon}, for $f\in\S$,
  \begin{align}
    \|g^w(v,D_v)f\|_{L^2}\le C_{k,K}\,\varepsilon\| ((b^{1/2})_{(K+\varepsilon^{-1-\delta})^{1/2},l^{1/2}})^wf\|_{L^2}.
  \end{align}
  Also, in the proof of corollary \ref{equaivlent_coro}, we have shown that $c^w$ is invertible in the form of \eqref{inverse_equation}.
  Thus for $f\in\S$,
  \begin{align*}
    &|([c^w,a^w]f,\,c^wf)_{L^2}|\\
    &=\left|\left([c^w,a^w](c^w)^{-1}c^wf,\,c^wf\right)_{L^2}\right|\\
    &=\left|\left(((b^{1/2})_{(K+\varepsilon^{-1-\delta})^{1/2},l^{1/2}}^w)^{-1}[c^w,a^w](c^w)^{-1}c^wf,\,(b^{1/2})_{(K+\varepsilon^{-1-\delta})^{1/2},l^{1/2}}^wc^wf\right)_{L^2}\right|\\
    &\le \varepsilon\, C_{k,K}\|(b^{1/2})_{(K+\varepsilon^{-1-\delta})^{1/2},l^{1/2}}^wc^wf\|^2_{L^2}\\
&\le \varepsilon C_{k,K}\Big(\|(b^{1/2})^wc^wf\|^2_{L^2}+(K+\varepsilon^{-1-\delta})^{1/2}\|(l^{1/2})^{1/2}c^wf\|^2_{L^2})\Big).
  \end{align*}
  By fixing $K>K^2_0$ and choosing $\varepsilon>0$ sufficiently small, we proved the claim.
\end{proof}


\section{Semigroup of linearized Boltzmann operator}\label{section3}
In this section, we will prove the main result \ref{Main2}.
To obtain a pseudo-differential form of the linearized Boltzmann operator, we will follow the argument in \cite{Global2019}.
Fix $0<\delta\le 1$. Let $\varphi(t)$ be a positive smooth radial function that equal to $1$ when $|t|\le 1/4$ and $0$ when $|t|\ge 1$.
Let $\varphi_\delta(v)=\varphi(|v|^2/\delta^2)$ and $\widetilde{\varphi}_\delta(v)=1-\varphi_\delta(v)$. Then $\varphi_\delta(v)=\varphi(|v|^2/\delta^2)$ equal to $0$ for $|v|\ge \delta$ and $1$ for $|v|\le \delta/2$.
Also, \cite{Global2019} has shown that
\begin{align}
  L&=-a^w-\left(-L_2-\widetilde{L}_{1,\delta,a}-L_{1,3,\delta}-L_{1,4,\delta}+a_s(v,D_v)+(a(v,D_v)-a^w(v,D_v))\right),
\end{align}where
\begin{align*}
  L_2f&:=\int\int B(\mu_*)^{1/2}\left((\mu')^{1/2}f'_*-\mu^{1/2}f_*\right)\,dv_*d\sigma,\\
  \widetilde{L}_{1,\delta,a}f &:= \int\int B\widetilde{\varphi}_\delta(v'-v)(\mu_*)^{1/2}(\mu'_*)^{1/2}f'\,dv_*d\sigma,\\
  L_{1,3,\delta}f&:=f(v)\int\int B \varphi_\delta(v'-v)(\mu'_*-\mu_*)\,dv_*d\sigma,\\
  L_{1,4,\delta}f&:=f(v)\int\int B \varphi_\delta(v'-v)(\mu'_*)^{1/2}((\mu_*)^{1/2}-(\mu'_*)^{1/2})\,dv_*d\sigma,\\
  a_s(v,D_v)f&:=-\int\int B\varphi_\delta(v'-v)(\mu'_*)^{1/2}(f'-f)((\mu_*)^{1/2}-(\mu'_*)^{1/2})\,dv_*d\sigma\\
  a(v,D_v)f&:=-\int\int B\varphi_\delta(v'-v)\mu'_*(f'-f)\,dv_*d\sigma \\
  &\qquad\qquad\qquad + f(v)\int\int B\widetilde{\varphi}_\delta(v'-v)\mu_*\,dv_*d\sigma,
\end{align*}where $a_s$ and $a$ can be written in the form of Carleman representation.
\begin{align*}
  a_s(v,\eta)&:=-\int_{\R^d_h}\int_{E_{0,h}}\tilde{b}\1_{|\alpha|\ge|h|}\varphi_\delta(h)\frac{|\alpha+h|^{1+\gamma+2s}}{|h|^{d+2s}}
  \mu^{1/2}(v+\alpha)\\
  &\qquad\qquad\qquad\qquad\qquad(e^{-2\pi ih\cdot\eta}-1)(\mu^{1/2}(v+\alpha-h)-\mu^{1/2}(v+\alpha))\,d\alpha dh \\
  a(v,\eta)&:=\int_{\R^d_h}\int_{E_{0,h}}\tilde{b}\1_{|\alpha|\ge|h|}\varphi_\delta(h)\frac{|\alpha+h|^{1+\gamma+2s}}{|h|^{d+2s}}\mu(v+\alpha)
  (1-\cos(2\pi\eta\cdot h))\,d\alpha dh \\
  &\qquad\quad+\int_{\R^d_h}\int_{E_{0,h}}\tilde{b}\1_{|\alpha|\ge|h|}\widetilde{\varphi}_\delta(h)\frac{|\alpha+h|^{1+\gamma+2s}}{|h|^{d+2s}}\mu(v+\alpha-h) \,d\alpha dh.
\end{align*}
Recall that
\begin{align}
\tilde{a}(v,\eta):=\<v\>^\gamma(1+|\eta|^2+|\eta\wedge v|^2+|v|^2)^s.
\end{align}
Proposition 3.7 in \cite{Global2019} shows that $\tilde{a}$ is a $\Gamma$-admissible weight, $\tilde{a}\approx a$ and
\begin{align}\label{temp22}
a,\,\tilde{a},\,a_s\in S(\tilde{a}),\ \ \partial_\eta a,\, \partial_\eta \tilde{a}\in S(\varepsilon \tilde{a}+\varepsilon^{-1}\<v\>^{\gamma+2s}),
\end{align}
uniformly in $\varepsilon$.

We define a symbol $b$ by
\begin{align}
  b^w(v,D_v):=a^w(v,D_v)+a_s(v,D_v)+(a(v,D_v)-a^w(v,D_v)).
\end{align}Then $b\in S(\tilde{a})$ and 
\begin{align}\label{equation_L}
  L
  &=-b^w(v,D_v)+\left(L_2+\widetilde{L}_{1,\delta,a}+L_{1,3,\delta}+L_{1,4,\delta}\right).
\end{align}

Firstly, we analyze the pseudo-differential part $b^w$. To apply the theorem \ref{Dissipation_Theorem}, we let
\begin{align}
  l(v):=\<v\>^{\gamma+2s}.
\end{align}Then $l\in S(l)$ and $l\le \tilde{a}$. Please notice that $a$ is positive while $b$ may not.
\begin{Thm}\label{Thm31}Assume $\gamma+2s\le 0$. There exists $C_1>0$ such that
  $-(C_1+b)^w:H(\tilde{a}_1c)\to H(c)$ generates a contraction semigroup on $H(c)$, with $\tilde{a}_1:=\tilde{a}+C_1$.
  Consequently, $-b^w:H(\tilde{a}_1c)\to H(c)$ generates a strongly continuous semigroup on $H(c)$.
\end{Thm}
\begin{proof}For any $C_1>1$, write $b_1:=C_1+b$ and $a_1:=C_1+a$. By semigroup theory \ref{semigroup},
  it suffices to show that there exists $C_1>1$ such that $(-b_1^w,D(b_1^w))$ is dissipative on Hilbert space $H^k_n$ with $D(b_1^w):=H(\tilde{a}_1c)$ and $\lambda I+b_1^w:H(\tilde{a}_1c)\to H(c)$ is invertible for some $\lambda>0$.
Notice here $\tilde{a}_1\ge 1$, hence the identity operator maps $H(\tilde{a}_1c)$ into $H(\tilde{a}_1c)\subset H(c)$ and so $\lambda I+b_1^w$ is well-defined.

To prove $b_1^w$ is dissipative on $H(c)$, we shall verify the assumptions in theorem \ref{Dissipation_Theorem}.
Let $a_d$ to be the symbol of $a^w(v,D_v)-a(v,D_v)$ as a Weyl quantization. Then by lemma 4.4 in \cite{Global2019}, we have
\begin{align}
 a_s,\, a_d,\, \partial_{\eta}\tilde{a},\,\partial_{\eta} a\in S(\varepsilon \tilde{a}+\varepsilon^{-1}\<v\>^{\gamma+2s}),
\end{align}uniformly in $\varepsilon$.
Thus $b=a+a_s+a_d\in S(\tilde{a})$ and $\partial_{\eta} b\in S(\varepsilon\tilde{a}+\varepsilon^{-1}\<v\>^{\gamma+2s})$ uniformly in $\varepsilon$.
So $b_{K,l}\in S(\tilde{a}_{K,l})$ and if we choose $\varepsilon = K^{-1/2}$, then we have
\begin{align}
  \partial_\eta (b_{K,l})\in S(K^{-1/2}\tilde{a}_{K,l})
\end{align}uniformly in $K$.
Thus $b$ satisfies assumption (1) in theorem \ref{Dissipation_Theorem} with $m=\tilde{a}$ and it's trivial that $\tilde{a}$ satisfies assumption (2) in theorem \ref{Dissipation_Theorem} with $m=\tilde{a}$ by using \eqref{temp22}.

On the other hand, $a\, ,\tilde{a}\in S(\tilde{a})$. Choosing $\varepsilon = K^{-1/2}$ in \eqref{temp22}, we find that $\partial_\eta a,\, \partial_\eta \tilde{a}\in S(K^{-1/2}\tilde{a}_{K,l})$. Since $a\approx \tilde{a}$, we know $a_{K,l}\gtrsim \tilde{a}_{K,l}$ and hence $a$ satisfies theorem \ref{Dissipation_theoremL2} with $m=\tilde{a}$.
Thus there exists $C>0$ such that for $f\in\S$,
\begin{align}\label{eqq37}
  \Re(a^w(v,D_v)f,f)_{L^2}\ge \frac{1}{C}\|(\tilde{a}^{1/2})^w(v,D_v)f\|^2_{L^2} - C\|\<v\>^{\gamma/2+s}f\|^2_{L^2}.
\end{align}
Since $\partial_\eta \tilde{a}\in S(\varepsilon\tilde{a}+\varepsilon^{-1}\<v\>^{\gamma+2s})$ uniformly in $\varepsilon\in(0,1)$ and $\tilde{a}\ge \<v\>^{\gamma+2s}$, we have $\partial_\eta(\tilde{a}^{1/2})_{K,l^{1/2}}\in S(\varepsilon\tilde{a}^{1/2}+\varepsilon^{-1}\<v\>^{\gamma/2+s})$ uniformly in $\varepsilon$. Hence choosing $\varepsilon=K^{-1/2}$, we know that $\tilde{a}^{1/2}$ satisfies lemma \ref{inverse_lemma} with $m=\tilde{a}$ and $l=\<v\>^{\gamma/2+s}$ therein.
Thus there exists $K_0>1$ such that for $K>K_0$, $(\tilde{a}^{1/2})_{K,l^{1/2}}:H(\tilde{a})\to L^2$ is invertible in the form of \eqref{inverse_equation} and $((\tilde{a}^{1/2})_{K,l^{1/2}})^{-1}\in S(((\tilde{a}^{1/2})_{K,l^{1/2}})^{-1})$.

As in \cite{Global2019}, we let $a_{pseudo}:=a_s+a_d\in  S(\varepsilon\tilde{a}_{K+\varepsilon^{-2},l})$ uniformly in $\varepsilon$.
Noticing
\begin{align*}
((\tilde{a}^{1/2})_{K+\varepsilon^{-2},l^{1/2}})^{-1}\#a_{pseudo}&\in
S\left( \frac{\varepsilon\big(\tilde{a}+(K+\varepsilon^{-2})l\big)}{\tilde{a}^{\frac{1}{2}}+(K+\varepsilon^{-2})l^{1/2}}\right)\\
&\subset S(\varepsilon(\tilde{a}^{1/2})_{K,l^{1/2}}+\varepsilon^{-1}\<v\>^{\gamma/2+s}),
\end{align*}uniformly in $\varepsilon$, we can apply lemma \ref{bound_varepsilon} to get
\begin{align*}
  |(a^w_{pseudo}f,f)_{L^2}| &= |(((\tilde{a}^{1/2})_{K,l^{1/2}}^w)^{-1}a^w_{pseudo}f,(\tilde{a}^{1/2})_{K,l^{1/2}}^wf)_{L^2}|\\
  &\le C(\varepsilon\|(\tilde{a}^{1/2})^w(v,D_v)f\|^2_{L^2}+C(K,\varepsilon)\|\<v\>^{\gamma/2+s}f\|^2_{L^2}).
\end{align*}
Then picking $\varepsilon$ small, we have for $f\in\S$, by \eqref{eqq37},
\begin{align*}
  \Re(b^w(v,D_v)f,f)_{L^2} &= \Re(a^w(v,D_v)f,f)_{L^2} + \Re(a^w_{pseudo}f,f)_{L^2}\\
  &\ge \frac{1}{C'}\|(\tilde{a}^{1/2})^w(v,D_v)f\|^2_{L^2} - C\|\<v\>^{\gamma/2+s}f\|^2_{L^2}.
\end{align*}

Now, all the assumptions in theorem \ref{Dissipation_Theorem} are fulfilled. Hence there exists $C_0\ge1$ such that for any $f\in\S$,
\begin{align}
  \Re(b^w(v,D_v)f,f)_{H^k_n}&\ge \frac{1}{C_0}\|(\tilde{a}^{1/2})^w(v,D_v)c^wf\|^2_{L^2}
  - C_0\|\<v\>^{\gamma/2+s}c^wf\|^2_{L^2}\\
  &\ge \frac{1}{C_0}\|(\tilde{a}^{1/2})^w(v,D_v)\<D_v\>^kf\|^2_{L^2}
  - C_0\|f\|_{H^k_n},
\end{align}since $\gamma+2s\le 0$.
Thus whenever $C_1>C_0\ge 1$, for $f\in\S$,
\begin{align}
  \Re((b_1)^w(v,D_v)f,f)_{H^k_n}\ge 0.
\end{align}
Recall that the domain of $b_1$ is $H(\tilde{a}_1c)\hookrightarrow H^k_n$. Thus the above inequality is valid for $f\in D(b_1)$, since $\S$ is dense in $H(\tilde{a}_1c)$.

  Now we let $l_1=1$, then $l_1\lesssim \tilde{a}_1$, $b_1\in S(\tilde{a}_1)$, $\partial_\eta((b_1)_{K,l_1})\in S(K^{-1/2}(\tilde{a}_1)_{K,l_1})$ uniformly in $K$.
  Since $a_s,a_d\in S(\varepsilon\tilde{a}+\varepsilon^{-1}\<v\>^{\gamma+2s})$, we choose $\varepsilon$ small enough, then
  \begin{align*}
    |C_1+b(v,\eta)+K|
    &\ge K+C_1+a(v,\eta)- |a_s(v,\eta)| - |a_d(v,\eta)|\\
    &\gtrsim K+C_1+\tilde{a}(v,\eta)-\frac{1}{2}\tilde{a}-C\<v\>^{\gamma+2s}
    \gtrsim K+C_1+\tilde{a},
  \end{align*}if $C_1>2C$.
  Thus fixing $C_1$ sufficiently large, $b_1(v,\eta)$ satisfies the assumption of lemma \ref{inverse_lemma} with $m=\tilde{a}_1$ and $l=l_1$ therein. Hence there exists sufficiently large $K_0$ such that for $\lambda>K_0$, $\lambda I+C_1I+b^w(v,D_v):H(\tilde{a}_1c)\to H(c)$ is invertible, hence is surjective.


Thus by \ref{semigroup}, $(-(b_1)^w,D(b_1))$ generates a contraction semigroup on $H(c)$. But $C_1I$ is a bounded perturbation on $H(c)$, hence $(-b^w(v,D_v),D(b_1))$ generates a strongly continuous semigroup on $H(c)$.
\end{proof}

\begin{proof}[Proof of theorem \ref{Main2}]
By \ref{Thm31} and \eqref{equation_L}, it suffices to prove the operator inside parentheses of \eqref{equation_L} is bounded on $H(c)$, which is shown in the next theorem \ref{Thm32}. Then our proof of \ref{Main2} is completed. 
\end{proof}
Next we will show that the operator in the parentheses of \eqref{equation_L} is actually Weyl quantization with symbol in $S(1)$. The idea here is to use Carleman representation.
\begin{Thm}\label{Thm32}

(1). The operators $L_{1,3,\delta}$, $\widetilde{L}_{1,\delta,a}$, $L_{1,4,\delta}$ and $L_2$ are all Weyl quantizations with symbols in $S(\<v\>^{\gamma+2s})$. 

(2). In particular, if $\gamma+2s\le 0$, their symbols belong to $S(1)$ and the corresponding Weyl quantizations are bounded on $H(c)$.
\end{Thm}
\begin{proof}Set $f\in\S$. For the part $L_{1,3,\delta}$, by lemma 2.3 in \cite{Global2019},
\begin{align}
  L_{1,3,\delta}f = S*_{v_*}\mu(v) f(v),
\end{align}with $S(z) = S_1(z)+S_2(z)$ satisfying
\begin{align*}|S_1(z)|&\le C|z|^\gamma,\qquad
  |S_2(z)|\le C\delta^{2-2s}|z|^{\gamma+2s-2}.
\end{align*}So by \eqref{equivalent_decay},
\begin{align*}
  |\partial^\alpha_v(S*_{v_*}\mu(v))| &= |S*_{v_*}(\partial^\alpha_v)\mu(v)|\le C(\<v\>^{\gamma}+\delta^{2-2s}\<v\>^{\gamma+2s-2}).
\end{align*}Thus fixing $\delta>0$, the symbol of $L_{1,3,\delta}$ belongs to $Op(\<v\>^{\gamma+2s})$.

 Now we turn to the non-singular part $\widetilde{L}_{1,\delta,a}$.
\begin{align*}
  \widetilde{L}_{1,\delta,a}f &= \int_{\R^d_h}\int_{E_{0,h}}\tilde{b}(\alpha,h)\1_{|\alpha|\ge |h|}\frac{|\alpha+h|^{\gamma+1+2s}}{|h|^{d+2s}}\widetilde{\varphi}_\delta(h)\\
  &\qquad\qquad\qquad\qquad\qquad\qquad\mu^{1/2}(v+\alpha-h)\mu^{1/2}(v+\alpha)f(v-h)\,d\alpha dh\\
  &=:\widetilde{a}_{1,\delta,a}(v,D_v)f,
\end{align*}with
\begin{align*}\widetilde{a}_{1,\delta,a}(v,\eta)&:=
  \int_{\R^d_h}\int_{E_{0,h}}\tilde{b}(\alpha,h)\1_{|\alpha|\ge |h|}\frac{|\alpha+h|^{\gamma+1+2s}}{|h|^{d+2s}}\widetilde{\varphi}_\delta(h)\\
  &\qquad\qquad\qquad\qquad\qquad\mu^{1/2}(v+\alpha-h)\mu^{1/2}(v+\alpha)e^{2\pi ih\cdot\eta}\,d\alpha dh.
\end{align*}Thus
\begin{align*}
  \partial^{\tilde{\beta}}_v\partial^\beta_\eta\tilde{a}_{1,\delta,a}
  &=\int_{\R^d_h}\int_{E_{0,h}}\tilde{b}(\alpha,h)\1_{|\alpha|\ge |h|} \1_{|h|\ge\delta/2} \frac{|\alpha+h|^{\gamma+1+2s}}{|h|^{d+2s}}\widetilde{\varphi}_\delta(h)\\
  &\qquad\qquad\qquad\qquad\qquad\partial^{\tilde{\beta}}_v(\mu^{1/2}(v+\alpha-h)\mu^{1/2}(v+\alpha)) \partial^\beta_\eta e^{2\pi ih\cdot\eta} \,d\alpha dh,\\
  |\partial^{\tilde{\beta}}_v\partial^\beta_\eta\tilde{a}_{1,\delta,a}|
  &\le C_{\tilde{\beta},\beta}\int_{\R^d_h}\int_{E_{0,h}}\1_{|\alpha|\ge |h|} \1_{|h|\ge\delta/2} \frac{|\alpha+h|^{\gamma+1+2s}}{|h|^{d+2s}}\widetilde{\varphi}_\delta(h)\\
  &\qquad\qquad\qquad\qquad\qquad\qquad
  \mu^{1/4}(v+\alpha-h)\mu^{1/4}(v+\alpha) |h|^{|\beta|} \,d\alpha dh\\
  &\le C_{\tilde{\beta},\beta}\int_{\R^d_h}\int_{E_{0,h}}\1_{|\alpha|\ge |h|} \1_{|h|\ge\delta/2} \frac{|\alpha+h|^{\gamma+1+2s}}{|h|^{d+2s}}\widetilde{\varphi}_\delta(h)
  \mu^{1/8}(v+\alpha-h) d\alpha dh\\
  &\le C_{\tilde{\beta},\beta}\<v\>^{\gamma+2s}.
\end{align*}
Thus $\tilde{a}_{1,\delta,a}\in S(\<v\>^{\gamma+2s})$. The last inequality follows from the argument of Proposition 3.5 in \cite{Global2019}, since it's the same as equation (38) therein.

For $L_{1,4,\delta}$, by Lemma 2.5 in \cite{Global2019}, we have
\begin{align}
  L_{1,4,\delta}f = -\frac{1}{2}L_{1,3,\delta}f - D(v)f,
\end{align}where
\begin{align*}
  D(v)&:=\frac{1}{2}\int\int B\varphi_\delta(v'-v)\left((\mu_*)^{1/2}-(\mu'_*)^{1/2}\right)^2\,dv_*d\sigma\\
  &=\frac{1}{2}\int_{\R^d_h}\int_{E_{0,h}}\tilde{b}(\alpha,h)\1_{|\alpha|\ge |h|}\frac{|\alpha+h|^{\gamma+1+2s}}{|h|^{d+2s}}\varphi_\delta(h)\\
  &\qquad\qquad\qquad\qquad\qquad\left(\mu^{1/2}(v+\alpha-h)-\mu^{1/2}(v+\alpha)\right)^2\,d\alpha dh.
  \end{align*}Hence by lemma \ref{derivative} below,
  \begin{align*}
  |\partial^\beta_vD(v)|
  &\le\frac{1}{2}\int_{\R^d_h}\int_{E_{0,h}} \1_{|\alpha|\ge |h|}\frac{|\alpha+h|^{\gamma+1+2s}}{|h|^{d+2s}}\varphi_\delta(h)\\
  &\qquad\qquad\qquad\qquad\qquad\left|\partial^\beta_v\left(\mu^{1/2}(v+\alpha-h)-\mu^{1/2}(v+\alpha)\right)^2\right|\,d\alpha dh,\\
  &\le C\int_{\R^d_h}\int_{E_{0,h}} \1_{|\alpha|\ge |h|}\frac{|\alpha+h|^{\gamma+1+2s}}{|h|^{d+2s}}\varphi_\delta(h)\mu^{1/8}(v+\alpha) |h|^2\,d\alpha dh,\\
  &\le C\delta^{2-2s}\<v\>^{\gamma+2s},
\end{align*}where the last step follows from Lemma 2.5 in \cite{Global2019}.

 Now we deal with the last term $L_2$.
\begin{align*}
  L_2f &= \int\int B(\mu_*)^{1/2}\left((\mu')^{1/2}f'_*-\mu^{1/2}f_*\right)\,dv_*d\sigma\\
  &= \int\int B\left((\mu^{1/2}f)'_*(\mu')^{1/2}-(\mu^{1/2}f)_*(\mu^{1/2})\right)\,dv_*d\sigma\\
  &\qquad\qquad + \int\int B(\mu')^{1/2}\left((\mu_*)^{1/2}-(\mu'_*)^{1/2}\right)f'_*\,dv_*d\sigma\\
  &= \int\int B(\mu^{1/2}f)'_*\left((\mu')^{1/2}-\mu^{1/2}\right)\,dv_*d\sigma\\
  &\qquad
  + \mu^{1/2}\int\int B\left((\mu^{1/2}f)'_*-(\mu^{1/2}f)_*\right)\,dv_*d\sigma\\
  &\qquad\qquad+ \mu^{1/2}\int\int B\left((\mu_*)^{1/2}-(\mu'_*)^{1/2}\right)f'_*\,dv_*d\sigma\\
  &\qquad\qquad\qquad+ \int\int B\left((\mu')^{1/2}-\mu^{1/2}\right)\left((\mu_*)^{1/2}-(\mu'_*)^{1/2}\right)f'_*\,dv_*d\sigma\\
  &=: L_{2,r}f+L_{2,ca}f+L_{2,c}f+L_{2,d}f.
\end{align*}
We will investigate these four parts separately. For $L_{2,ca}$, by Cancellation Lemma, there exists constant $C$ depending only on $B$ hence only on $s$ such that
\begin{align*}
  L_{2,ca}f &= C\mu^{1/2}\int_\Rd |v-v_*|^\gamma(\mu^{1/2}f)_*\,dv_*\\
  &= C\mu^{1/2}\int_\Rd |v_*|^\gamma\mu^{1/2}(v_*+v)f(v_*+v)\,dv_*\\
  &= C\mu^{1/2}\int_\Rd |v_*|^\gamma\mu^{1/2}(v_*+v)\int_\Rd\widehat{f}(\eta)e^{2\pi i (v+v_*)\cdot\eta}\,d\eta\,dv_*\\
  &=: a_{2,ca}(v,D_v)f,
\end{align*}with
\begin{align*}
  a_{2,ca}(v,\eta)=C\mu^{1/2}\int_\Rd |v_*|^\gamma\mu^{1/2}(v_*+v)e^{2\pi i v_*\cdot\eta}\,dv_*.
\end{align*}Then
\begin{align*}
  |\partial^\alpha_v\partial^\beta_\eta a_{2,ca}(v,\eta)| &\le C_{\alpha,\beta}\mu^{1/4}(v)\int_\Rd|v_*|^\gamma\mu^{1/4}(v_*+v)|v_*|^{|\beta|}\,dv_*\\
  &\le C_{\alpha,\beta}\mu^{1/4}(v)\<v\>^{\gamma+|\beta|}\\
  &\le C_{\alpha,\beta}\mu^{1/8}(v)\le C_{\alpha,\beta,b}\<v\>^{\gamma+2s}.
\end{align*}
Thus $a_{2,ca}\in S(\<v\>^{\gamma+2s})$.
For $L_{2,c}$, by using Carleman representation,
\begin{align*}
  L_{2,c} &= \mu^{1/2}\int\int B\left((\mu_*)^{1/2}-(\mu'_*)^{1/2}\right)f'_*\,dv_*d\sigma\\
  &= \int_\Rd\int_{E_{0,h}}\widetilde{b}(\alpha,h)\1_{|\alpha|\ge|h|}\frac{|\alpha+h|^{\gamma+1+2s}}{|h|^{d+2s}}\mu^{1/2}(v)\\
  &\qquad\qquad\qquad\qquad\left(\mu^{1/2}(v+\alpha-h)-\mu^{1/2}(v+\alpha)\right)f(v+\alpha)\,d\alpha dh\\
  &=: a_{2,c}(v,D_v)f,
\end{align*}with
\begin{align*}
  a_{2,c}(v,\eta)
  &:=  \int_\Rd\int_{E_{0,h}}\widetilde{b}(\alpha,h)\1_{|\alpha|\ge|h|}\frac{|\alpha+h|^{\gamma+1+2s}}{|h|^{d+2s}}\mu^{1/2}(v)\\
  &\qquad\qquad\qquad\qquad
  \left(\mu^{1/2}(v+\alpha-h)-\mu^{1/2}(v+\alpha)\right)e^{2\pi i \alpha\cdot\eta}\,d\alpha dh\\
  &= \frac{1}{2} \int_\Rd\int_{E_{0,h}}\widetilde{b}(\alpha,h)\1_{|\alpha|\ge|h|}\frac{|\alpha+h|^{\gamma+1+2s}}{|h|^{d+2s}}\mu^{1/2}(v)\\&\qquad
  \left(\mu^{1/2}(v+\alpha-h)+\mu^{1/2}(v+\alpha+h)-2\mu^{1/2}(v+\alpha)\right)e^{2\pi i \alpha\cdot\eta}\,d\alpha dh.
\end{align*}
We split this integral into two parts: $\1_{|h|\ge 1}$ and $\1_{|h|\le 1}$, the non-singular part and singular part. Then for any multi-index $\tilde{\beta},\beta \in \N^d$,
\begin{align*}
  &|\partial^{\tilde{\beta}}_v\partial^\beta_\eta a_{2,c,non-singular}|\\
  &= \Big|\partial^{\tilde{\beta}}_v\partial^\beta_\eta\int_\Rd\int_{E_{0,h}} \widetilde{b}(\alpha,h)\1_{|\alpha|\ge|h|}\1_{|h|\ge 1} \frac{|\alpha+h|^{\gamma+1+2s}}{|h|^{d+2s}}\mu^{1/2}(v)\\
  &\qquad\qquad\qquad\qquad\qquad
  \left(\mu^{1/2}(v+\alpha-h)-\mu^{1/2}(v+\alpha)\right)e^{2\pi i \alpha\cdot\eta}\,d\alpha dh\Big|\\
  &\le C_{\tilde{\beta},\beta} \int_\Rd\int_{E_{0,h}}\1_{|\alpha|\ge|h|}\1_{|h|\ge 1}\frac{|\alpha+h|^{\gamma+1+2s}}{|h|^{d+2s}} \mu^{1/4}(v)\\
  &\qquad\qquad\qquad\qquad\qquad
  \left(\mu^{1/4}(v+\alpha-h)+\mu^{1/4}(v+\alpha)\right)|\alpha|^{|\beta|}\,d\alpha dh\\
  &\le C_{\tilde{\beta},\beta} \int_\Rd\int_{E_{0,h}}\1_{|\alpha|\ge|h|}\1_{|h|\ge 1}
  \frac{1}{|h|^{d+2s}}
  \bigg(|\alpha+h|^{\gamma+1+2s+|\beta|}\mu^{1/4}(v)\mu^{1/4}(v+\alpha-h)\\
  &\qquad\qquad\qquad\qquad\qquad\qquad\qquad\qquad\
   +|\alpha|^{\gamma+1+2s+|\beta|}\mu^{1/4}(v)\mu^{1/4}(v+\alpha)\bigg)\,d\alpha dh\\
  &\le C_{\tilde{\beta},\beta} \int_\Rd\int_{E_{0,h}}\1_{|\alpha|\ge|h|}\1_{|h|\ge 1}\frac{1}{|h|^{d+2s}}
  \mu^{1/72}(v)\mu^{1/72}(v+\alpha)\,d\alpha dh\\
  &\le C_{\tilde{\beta},\beta} \mu^{1/72}(v)\le C_{\tilde{\beta},\beta}\<v\>^{\gamma+2s},
\end{align*}by using the lemma \ref{v_and_v_star} below.
For the singular part, applying lemma \ref{derivative} below, we have
\begin{align*}
  &|\partial^{\tilde{\beta}}_v\partial^\beta_\eta a_{2,c,singular}|\\
  &= |\partial^{\tilde{\beta}}_v\partial^\beta_\eta\frac{1}{2} \int_\Rd\int_{E_{0,h}}\widetilde{b}(\alpha,h)\1_{|\alpha|\ge|h|}\1_{|h|\le 1} \frac{|\alpha+h|^{\gamma+1+2s}}{|h|^{d+2s}}\mu^{1/2}(v)\\
  &\qquad\qquad\qquad\qquad
  \left(\mu^{1/2}(v+\alpha-h)+\mu^{1/2}(v+\alpha+h)-2\mu^{1/2}(v+\alpha)\right)e^{2\pi i \alpha\cdot\eta}\,d\alpha dh|\\
  &\le C_{\tilde{\beta},\beta}\int_\Rd\int_{E_{0,h}}\1_{|\alpha|\ge|h|}\1_{|h|\le 1}\frac{|\alpha+h|^{\gamma+1+2s}}{|h|^{d+2s}} |h|^2\mu^{1/4}(v)\mu^{1/16}(v+\alpha)|\alpha|^{|\beta|}\,d\alpha dh,\\
  &\le C_{\tilde{\beta},\beta}\int_\Rd\int_{E_{0,h}}\1_{|\alpha|\ge|h|}\1_{|h|\le 1}\frac{1}{|h|^{d+2s-2}} \mu^{1/32}(v)\mu^{1/32}(v+\alpha)\,d\alpha dh,\\
  &\le C_{\tilde{\beta},\beta}\mu^{1/32}(v),
\end{align*}
Since $s\in(0,1)$ and $\frac{1}{|h|^{d+2s-2}}$ is locally integrable on $\Rd$. Thus $a_{2,c}\in S(\<v\>^{\gamma+2s})$. For the part $L_{2,r}$, the argument is similar.
\begin{align*}
  L_{2,r}f &= \int\int B (\mu^{1/2}f)'_*((\mu^{1/2})'-\mu^{1/2})\,dv_*d\sigma\\
  &=\int_\Rd\int_{E_{0,h}}\widetilde{b}(\alpha,h)\1_{|\alpha|\ge|h|}\frac{|\alpha+h|^{\gamma+1+2s}}{|h|^{d+2s}}\mu^{1/2}(v+\alpha)\\
  &\qquad\qquad\qquad\qquad\left(\mu^{1/2}(v-h)-\mu^{1/2}(v)\right)f(v+\alpha)\,d\alpha dh\\
  &=: a_{2,r}(v,D_v)f,
\end{align*}with
\begin{align*}
  a_{2,r}(v,\eta)&:=
  \int_\Rd\int_{E_{0,h}}\widetilde{b}(\alpha,h)\1_{|\alpha|\ge|h|}\frac{|\alpha+h|^{\gamma+1+2s}}{|h|^{d+2s}}\mu^{1/2}(v+\alpha)\\
  &\qquad\qquad\qquad\qquad\left(\mu^{1/2}(v-h)-\mu^{1/2}(v)\right)e^{2\pi i\alpha\cdot\eta}\,d\alpha dh\\
  &=\frac{1}{2}\int_\Rd\int_{E_{0,h}}\widetilde{b}(\alpha,h)\1_{|\alpha|\ge|h|}\frac{|\alpha+h|^{\gamma+1+2s}}{|h|^{d+2s}}\mu^{1/2}(v+\alpha)\\&\qquad\qquad\qquad\qquad
  \left(\mu^{1/2}(v-h)+\mu^{1/2}(v+h)-2\mu^{1/2}(v)\right)e^{2\pi i\alpha\cdot\eta}\,d\alpha dh.
\end{align*}We split the integral into singular and non-singular part, and then the argument is similar to the part $L_{2,c}$ and we can obtain $a_{2,r}\in S(\<v\>^{\gamma+2s})$.
It remains to study $L_{2,d}$ which is
\begin{align*}
  L_{2,d} &= \int\int B\left((\mu')^{1/2}-\mu^{1/2}\right)\left((\mu_*)^{1/2}-(\mu'_*)^{1/2}\right)f'_*\,dv_*d\sigma\\
  &= \int_\Rd\int_{E_{0,h}}\widetilde{b}(\alpha,h)\1_{|\alpha|\ge|h|}\frac{|\alpha+h|^{\gamma+1+2s}}{|h|^{d+2s}}\left(\mu^{1/2}(v-h)-\mu^{1/2}(v)\right)\\
  &\qquad\qquad\qquad\qquad\qquad\qquad
  \left(\mu^{1/2}(v+\alpha-h)-\mu^{1/2}(v+\alpha)\right)f(v+\alpha)\,d\alpha dh\\
  &=a_{2,d}(v,D_v)f,
\end{align*}with
\begin{align*}
  a_{2,d}(v,\eta) &:=  \int_\Rd\int_{E_{0,h}}\widetilde{b}(\alpha,h)\1_{|\alpha|\ge|h|}\frac{|\alpha+h|^{\gamma+1+2s}}{|h|^{d+2s}}
  \left(\mu^{1/2}(v-h)-\mu^{1/2}(v)\right)\\
  &\qquad\qquad\qquad\qquad\qquad
  \left(\mu^{1/2}(v+\alpha-h)-\mu^{1/2}(v+\alpha)\right)
  e^{2\pi i \alpha\cdot\eta}\,d\alpha dh.
\end{align*}
Now using the identity $a^2-b^2=(a+b)(a-b)$ and lemma \ref{v_and_v_star}, we can split the Gaussian function into
\begin{align*}
	&\left(\mu^{1/2}(v-h)-\mu^{1/2}(v)\right)\left(\mu^{1/2}(v+\alpha-h)-\mu^{1/2}(v+\alpha)\right)\\
	&\quad =\mu^{1/80}(v)\mu^{1/80}(v+\alpha)\left(\mu^{1/4}(v-h)-\mu^{1/4}(v)\right)\left(\mu^{1/4}(v+\alpha-h)-\mu^{1/4}(v+\alpha)\right).
\end{align*} Then the remaining analysis is exactly the same as before. That is to split the integral into singular and non-singular parts. The terms inside the parentheses will cancel the singularity on $h$ and then we can have $a_{2,d}\in S(\<v\>^{\gamma+2s})$.
\end{proof}

Here we list two short lemmas used in the proof.
\begin{Lem}\label{v_and_v_star}If $|\alpha|\ge |h|$, $\alpha\cdot h=0$ then
  \begin{align*}
    \mu(v-h)\mu(v+\alpha)=\mu(v)\mu(v+\alpha-h)\le \mu^{1/9}(v)\mu^{1/9}(v+\alpha),\\
    \mu(v-h)\mu(v+\alpha-h)\le \mu^{1/20}(v)\mu^{1/20}(v+\alpha).
  \end{align*}
\end{Lem}
\begin{proof}
  Since $\alpha\cdot h=0$, we have $|v-h|^2+|v+\alpha|^2=|v|^2+|v+\alpha-h|^2$ and the first equality if proved.
  Notice $|v+\alpha|\le |v-h|+|\alpha+h|\le |v-h|+\sqrt{2}|\alpha|\le (1+\sqrt{2})|v-h|+\sqrt{2}|v+\alpha-h|$ and $|v|\le |v-h|+|h|\le |v-h|+|\alpha|\le 2|v-h|+|v+\alpha-h|$, we have
  \begin{align*}
    |v|^2+|v+\alpha|^2\le 20(|v-h|^2+|v+\alpha-h|^2),
  \end{align*}and the second inequality is proved. Similarly, $|v+\alpha|\le |v|+|\alpha-h|\le 2|v|+|v+\alpha-h|$ and hence
  \begin{align*}
    |v|^2+|v+\alpha|^2\le |v|^2+8|v|^2+2|v+\alpha-h|^2\le 9(|v|^2+|v+\alpha-h|^2),
  \end{align*}and hence the first inequality is proved.
\end{proof}

\begin{Lem}\label{derivative}If $|h|\le 1$, for $\beta\in\N^d$, there exists $C_{\beta}>0$ such that for $v\in \R^d$,
  \begin{align*}
    |\partial^\beta_v\left(\mu^{1/2}(v-h)-\mu^{1/2}(v)\right)|&\le C_{\beta}|h|\mu^{1/16}(v),\\
    |\partial^\beta_v\left(\mu^{1/2}(v-h)+\mu^{1/2}(v+h)-2\mu^{1/2}(v)\right)| &\le C_\beta|h|^2\mu^{1/16}(v).
  \end{align*}
\end{Lem}
\begin{proof}
  We recall the definition $\mu(v)=(2\pi)^{-d/2}e^{-|v|^2/2}$. Firstly, by mean value theorem, we have for some $\delta\in(0,1)$,
  \begin{align*}
    \left|\left(\mu^{1/2}(v-h)-\mu^{1/2}(v)\right)\right|
    &\le |h\partial_v(\mu^{1/2})(v-\delta h)| \le C|h|\mu^{1/4}.
  \end{align*}
  Thus
  \begin{align*}
    \left|\left(e^{-|v-h|^2/4}-e^{-|v|^2/4}\right)\right| &\le C |h|e^{-|v|^2/8},\\
    \left|\left(e^{v\cdot h/2-|h|^2/4}-1\right)\right| &\le C |h|e^{|v|^2/8},
  \end{align*}
  Now notice $\partial^\beta_v\left(\mu^{1/2}(v-h)-\mu^{1/2}(v)\right)
    = C \partial^\beta_v\big(e^{-|v|^2/4}(e^{v\cdot h/2-|h|^2/4}-1)\big)$,
  then the first estimate follows from Leibniz formula. The second inequality follows similarly.
  \end{proof}

\section{Appendix}
\paragraph{Pseudo-differential calculus}

  We recall some notation and theorem of pseudo-differential calculus. For details, one may refer to Chapter 2 in the book \cite{Lerner2010}, Proposition 1.1 in \cite{Bony1998-1999} and \cite{Beals1981,Bony1994} for details. As above, we set $\Gamma=|dv|^2+|d\eta|^2$, but notice that the following are also valid for general admissible metric.
  Let $M$ be an $\Gamma$-admissible weight function. That is, $M:\R^{2d}\to (0,+\infty)$ satisfies the following conditions:\\
  (a). (slowly varying) there exists $\delta>0$ such that for any $X,Y\in\R^{2d}$, $|X-Y|\le \delta$ implies
  \begin{align}
    M(X)\approx M(Y);
  \end{align}
  (b) (temperance) there exists $C>0$, $N\in\R$, such that for $X,Y\in \R^{2d}$,
  \begin{align}
    \frac{M(X)}{M(Y)}\le C\<X-Y\>^N.
  \end{align}
  A direct result is that if $M_1,M_2$ are two $\Gamma$-admissible weight, then so is $M_1+M_2$ and $M_1M_2$. Consider symbols $a(v,\eta,\xi)$ as a function of $(v,\eta)$ with parameters $\xi$. We say that
  $a\in S(M,\Gamma)$ uniformly in $\xi$, if for $\alpha,\beta\in \N^d$, $v,\eta\in\Rd$,
  \begin{align}
    |\partial^\alpha_v\partial^\beta_\eta a(v,\eta,\xi)|\le C_{\alpha,\beta}M,
  \end{align}with $C_{\alpha,\beta}$ a constant depending only on $\alpha$ and $\beta$, but independent of $\xi$. The space $S(M,\Gamma)$ endowed with the seminorms
  \begin{align}
    \|a\|_{k;S(M,\Gamma)} = \max_{0\le|\alpha|+|\beta|\le k}\sup_{(v,\eta)\in\R^{2d}}
    |M(v,\eta)^{-1}\partial^\alpha_v\partial^\beta_\eta a(v,\eta,\xi)|,
  \end{align}becomes a Fr\'{e}chet space.
  Sometimes we write $\partial_\eta a\in S(M,\Gamma)$ to mean that $\partial_{\eta_j} a\in S(M,\Gamma)$ $(1\le j\le d)$ equipped with the same seminorms.
  We formally define the pseudo-differential operator by
  \begin{align*}
    (op_ta)u(x)=\int_\Rd\int_\Rd e^{2\pi i (x-y)\cdot\xi}a((1-t)x+ty,\xi)u(y)\,dyd\xi,
  \end{align*}for $t\in\R$, $f\in\S$.
  In particular, denote $a(v,D_v)=op_0a$ to be the standard pseudo-differential operator and
  $a^w(v,D_v)=op_{1/2}a$ to be the Weyl quantization of symbol $a$. We write $A\in Op(M,\Gamma)$ to represent that $A$ is a Weyl quantization with symbol belongs to class $S(M,\Gamma)$. One important property for Weyl quantization of a real-valued symbol is the formal self-adjointness on $L^2$. Here, formal means the equation for self-adjointness is valid once they are well-defined.

Let $a_1(v,\eta)\in S(M_1,\Gamma),a_2(v,\eta)\in S(M_2,\Gamma)$, then $a_1^wa_2^w=(a_1\#a_2)^w$, $a_1\#a_2\in S(M_1M_2,\Gamma)$ with
\begin{align*}
  a_1\#a_2(v,\eta)=a_1(v,\eta)a_2(v,\eta)
  +\int^1_0(\partial_{\eta}a_1\#_\theta \partial_{v} a_2-\partial_{v} a_1\#_\theta \partial_{\eta} a_2)\,d\theta,\\
  g\#_\theta h(Y):=\frac{1}{(\pi\theta)^{2d}}\int_\Rd\int_\Rd e^{-2i\sigma(Y-Y_1,Y-Y_2)/\theta}g(Y_1) h(Y_2)\,dY_1dY_2,
\end{align*}with $Y=(v,\eta)$.
For any non-negative integer $k$, there exists $l,C$ independent of $\theta\in[0,1]$ such that
\begin{align}\label{sharp_theta}
  \|g\#_\theta h\|_{k;S(M_1M_2,\Gamma)}\le C\|g\|_{l,S(M_1,\Gamma)}\|h\|_{l,S(M_2,\Gamma)}.
\end{align}
Thus if $\partial_{\eta}a_1,\partial_{\eta}a_2\in S(M'_1,\Gamma)$ and $\partial_{v}a_1,\partial_{v}a_2\in S(M'_2,\Gamma)$, then $[a_1,a_2]\in S(M'_1M'_2,\Gamma)$, where $[\cdot,\cdot]$ is the commutator defined by $[A,B]:=AB-BA$.

We can define a Hilbert space $H(M,g):=\{u\in\S':\|u\|_{H(M,g)}<\infty\}$, where
\begin{align*}
  \|u\|_{H(M,g)}:=\int M(Y)^2\|\varphi^w_Yu\|^2_{L^2}|g_Y|^{1/2}\,dY<\infty,
\end{align*}and $(\varphi_Y)_{Y\in\R^{2d}}$ is any uniformly confined family of symbols which is a partition of unity. If $a\in S(M)$ is a isomorphism from $H(M')$ to $H(M'M^{-1})$, then $(a^wu,a^wv)$ is an equivalent Hilbertian structure on $H(M)$. Moreover, the space $\S(\Rd)$ is dense in $H(M)$.

For $0\le\delta\le\rho\le 1$, $\delta<1$, $m\in\R$, the metric $g_{\rho,\delta}:=\<\xi\>^{2\delta}|dx|^2 +\<\xi\>^{-2\rho}|d\xi|^2$ is admissible and
\begin{align*}
  H^m = H(\<\xi\>^{m},g_{1,0})=H(\<\xi\>^m,g_{\rho,\delta}).
\end{align*} This can be proved by using the technique in corollary \ref{equaivlent_coro}.

Let $a\in S(M,g)$, then
    $a^w:H(M_1,g)\to H(M_1/M,g)$ is linear continuous, in the sense of unique bounded extension from $\S$ to $H(M_1,\Gamma)$.
Also the existence of $b\in S(M^{-1},\Gamma)$ such that $b\#a = a\#b = 1$ is equivalent to the invertibility of $a^w$ as an operator from $H(MM_1,\Gamma)$
onto $H(M_1,\Gamma)$ for some $\Gamma$-admissible weight function $M_1$.

For the metric $\Gamma=|dv|^2+|d\eta|^2$, the map $J^t=\exp(2\pi i D_v\cdot D_\eta)$ is an isomorphism of the Fr\'{e}chet space $S(M,\Gamma)$, with polynomial bounds in the real variable $t$, where $D_v=\partial_v/i$, $D_\eta=\partial_\eta/i$. Moreover, $a(x,D_v)=(J^{-1/2}a)^w$.

\paragraph{Carleman representation and cancellation lemma}

Now we have a short review of some useful facts in the theory of Boltzmann equation. One may refer to \cite{Alexandre2000,Global2019} for details. The first one is the so called Carleman representation. For measurable function $F(v,v_*,v',v'_*)$, if any sides of the following equation is well-defined, then
\begin{align*}
  &\int_{\R^d}\int_{\mathbf{S}^{d-1}}b(\cos\theta)|v-v_*|^\gamma F(v,v_*,v',v'_*)\,d\sigma dv_*\\
  &\qquad=\int_{\R^d_h}\int_{E_{0,h}}\tilde{b}(\alpha,h)\1_{|\alpha|\ge|h|}\frac{|\alpha+h|^{\gamma+1+2s}}{|h|^{d+2s}}F(v,v+\alpha-h,v-h,v+\alpha)\,d\alpha dh,
\end{align*}where $\tilde{b}(\alpha,h)$ is bounded from below and above by positive constants, and $\tilde{b}(\alpha,h)=\tilde{b}(\pm\alpha,\pm h)$, $E_{0,h}$ is the hyper-plane orthogonal to $h$ containing the origin. The second is the cancellation lemma. Consider a measurable function $G(|v-v_*|,|v-v'|)$, then for $f\in\S$,
\begin{align*}
  \int_{\R^d}\int_{\mathbf{S}^{d-1}}G(|v-v_*|,|v-v'|)b(\cos\theta)(f'_*-f_*)\,d\sigma dv_* = S*_{v_*}f(v),
\end{align*}where $S$ is defined by, for $z\in\R^d$,
\begin{align*}
  S(z)=2\pi \int^{\pi/2}_0 b(\cos\theta)\sin\theta\left(G(\frac{|z|}{\cos\theta/2},\frac{|z|\sin\theta/2}{\cos\theta/2}) - G(|z|,|z|\sin(\theta/2))\right)\,d\theta.
\end{align*}

\paragraph{Semigroup theory}

Here we write some well-known result from semigroup theory. One may refer to \cite{Engel1999} for more details.
\begin{Def}
  A linear operator $(A,D(A))$ on a Banach space $X$ is called dissipative if
  $\|(\lambda I - A)x\|\ge \lambda \|x\|$
  for all $\lambda > 0$ and $x\in D(A)$.
\end{Def}
\begin{Prop}
  An operator $(A,D(A))$ is dissipative if and only if for every $x \in D(A)$ there exists $j(x) \in \{x'\in X':\<x,x'\>=\|x\|^2=\|x'\|^2\}$ such that
\begin{align}
  \Re\<Ax,j(x)\>\le 0.
\end{align}
\end{Prop}
\begin{Thm}
  For a densely defined, dissipative operator $(A,D(A))$ on a Banach space $X$ the following statements are equivalent.\\
(a) The closure $\overline{A}$ of $A$ generates a contraction semigroup.\\
(b) $Im(\lambda I - A)$ is dense in $X$ for some (hence all) $\lambda>0$.
\end{Thm}
\begin{Coro}\label{semigroup}
  Let $(A,D(A))$ be a dissipative operator on a reflexive Banach space such that $\lambda I-A$ is surjective for some $\lambda > 0$. Then $A$ is densely defined and
generates a contraction semigroup.
\end{Coro}
\begin{Thm}
  Let $(A,D(A))$ be the generator of a strongly continuous semigroup
  $(T(t))_{t\ge 0}$ on a Banach space $X$
  satisfying $\|T(t)\|\le Me^{\omega t}$ for all $t \ge 0$
  and some $\omega\in\R$, $M\ge 1$. If $B\in L(X)$, then
  $C := A + B$ with $D(C) := D(A)$
  generates a strongly continuous semigroup $(S(t))_{t\ge 0}$ satisfying
  $\|S(t)\|\le Me^{(\omega+M\|B\|)t}$ for all $t \ge 0$.
\end{Thm}

In the end, we write two useful inequality: for any $n_1<n_2<n_3$,
\begin{align}\label{varepsilon_inequality}
\<v\>^{n_2}\le \varepsilon\<v\>^{n_3}+C_{n_1,n_2,n_3}\varepsilon^{-\frac{n_2-n_1}{n_3-n_2}}\<v\>^{n_1}.
\end{align}For $\rho>0$, $\delta\in\R$, $\alpha>-d$, $\beta\in\R$, we have
\begin{align}\label{equivalent_decay}
	\int_{\Rd}|v|^\alpha\<v\>^\beta\<v+u\>^\delta e^{-\rho|v+u|^2}\,dv\approx \<u\>^{\alpha+\beta},
\end{align}where constants may depend on the parameters. The first one is a version of Young's inequality while the second is lemma 2.5 in \cite{Alexandre2012}.

\bibliographystyle{plainurl}
\bibliography{1.bib}
\end{document}